\documentclass[final,1p,times]{article}
\newcommand{\f}{\frac}
\newcommand{\p}{\partial}

\usepackage[utf8]{inputenc}
\usepackage{amssymb,amsbsy,amsmath,amsfonts,amssymb,amscd,amsthm,mathabx,stmaryrd}

\usepackage{amssymb,color}
\usepackage{epsfig}
\usepackage{latexsym,hyperref}
\usepackage{amssymb,amsbsy,amsmath,amsfonts,amssymb,amscd,amsthm,mathabx,stmaryrd}
\usepackage{epsfig,graphicx,subfigure}

\usepackage{pdfsync}
\usepackage{psfrag}
\usepackage{colortbl}

\usepackage[colorinlistoftodos,bordercolor=orange,backgroundcolor=orange!20,linecolor=orange,textsize=scriptsize]{todonotes}

\newtheorem{prop}{Proposition}

\DeclareMathOperator{\R}{{\mathbb R}}
\DeclareMathOperator{\N}{{\mathbb N}}
\DeclareMathOperator{\Z}{{\mathbb Z}}
\begin{document}

\title{Explicit solution and fine asymptotics for a critical growth-fragmentation equation}
\author{Marie Doumic\thanks{Sorbonne Universités, Inria, UPMC Univ Paris 06, Lab. J.L. Lions  UMR CNRS 7598, Paris, France} \thanks{Wolfgang Pauli Institute, University of Vienna, Vienna, Austria, marie.doumic@inria.fr}
\and Bruce van Brunt\thanks{Institute of Fundamental Sciences, Massey University, New Zealand,  B.vanBrunt@massey.ac.nz}}
%
%

%
\maketitle

\begin{abstract} We give here an explicit formula for the following critical case of the growth-fragmentation equation
$$\f{\p}{\p t} u(t,x)  +  \f{\p}{\p x} (g x u(t,x))+bu(t,x)=b \alpha^2  u(t,\alpha x),\qquad u(0,x)=u^0(x),$$
for some constants $g>0,$ $b>0$ and $\alpha>1$ - the case $\alpha=2$ being the emblematic binary fission case.
 We discuss the links between this formula and the asymptotic ones previously obtained in~\cite{DE}, and use them to clarify how periodicity may appear asymptotically. 
 \end{abstract}

\section*{Introduction}
Growth-fragmentation equations appear in many applications, ranging from protein polymerisation to internet protocols or cell division equation. Under a fairly general form it may be written as follows
$$\f{\p}{\p t} u(t,x) + \f{\p}{\p x} \big(g (x) u (t,x)\big) + B(x)u(t,x)=\int\limits_x^\infty k(y,x) B(y) u(t,y) dy,
$$
where $u(t,x)$ represents the concentration of individuals of size $x$ at time $t,$ $g$ their growth speed, $B$ the total instantaneous fragmentation probability rate and $k(y,x)$ the fragmentation probability of fragmenting individuals of size $y$ to give rise to individuals of size $x.$ Under assumptions linking fragmentation and growth parameters $B,$ $k$ and $g$, a steady asymptotic behaviour appears, \emph{i.e.} there exists a unique couple $(\lambda,U)$ with $\lambda>0$ such that $u(t,x) e^{-\lambda t} \to U(x)$ - see for instance the pioneering papers~\cite{DiekmannHeijmansThieme1984,HallWake_1989},~\cite{BP} for an introduction and many other references like~\cite{Zaidi_2015,Der4,Zai} for some most recent ones. {\color{black} This asymptotic behaviour is a key property of many models in the field of structured population dynamics, and in many cases such as bacterial growth it is experimentally observed~\cite{robert:hal-00981312} (the biologists speak of "desyncrhonization effect")}.However, such a steady behaviour may also fail for two types of reasons:
\begin{enumerate}
\item the balance assumptions between $B,$ $g$ and $k$ are not satisfied,
\item Growth and fragmentation are such that there is a lack of dissipativity in the equation.
This is typically the case when the growth is exponential, \emph{i.e.} $g(x)=gx,$ and the fragmentation is a dirac, $k(y,x)=\alpha\delta_{\f{x}{y}=\f{1}{\alpha}}$  with $\alpha>1.$ In such a case, if the division rate $B$ is such that
there exists a  positive couple $(\lambda,U)$, there also exists a countable set of complex couples of the form $(\lambda+i\theta_k, U_k),$ which leads to a periodic limit cycle, see~\cite{GreinerNagel,bernard:hal-01363549}. 
\end{enumerate}
We focus here on a critical case where both these reasons appear, namely $B(x)\equiv b>0$ - also called "homogeneous fragmentation" - $g(x)\equiv x g$ and $k(y,x)=\alpha\delta_{\f{x}{y}=\f{1}{\alpha}},$ which generalises the binary fission case $\alpha=2$. The equation under study is thus
 \begin{equation}\label{eq:main}
\f{\p}{\p t} u(t,x)  +  \f{\p}{\p x} (g x u(t,x))+bu(t,x)=b \alpha^2  u(t,\alpha x),\qquad u(0,x)=u^0(x).
\end{equation}
This is a specific case of the homogeneous fragmentation equation studied in~~\cite{MR2017852,BertoinWatson,DE}. { It may also be seen as an emblematic case when modelling bacterial growth, since the exponential growth in size of bacteria has been observed, together with equal mitosis ($\alpha=2$). A constant division rate would then correspond to a growth independent of the size. However, the behaviour that we study in this paper is barely observed in nature, since a tiny variability in the coefficient rates or in the fragmentation kernel is sufficient to drive the system towards a steady asymptotic growth. It is thus important for modellers to include such a slight variability rather than using directly the idealised model under study.} The main results obtained in~\cite{DE} were the following:
\begin{itemize}
\item a formulation in terms of Mellin and inverse Mellin transform was obtained, as soon as the initial condition $u_0$ decays sufficiently fast in $0$ and $\infty$,
\item  no steady or self-similar behaviour was possible for $L^1$ functions,
\item the asymptotic behaviour was described along lines of the type $x=e^{-ct},$ with an exponential speed of convergence at places where the mass was decaying, but with at most polynomial growth for the lines where the mass concentrates,
\item in the case of a fragmentation kernel defined as a dirac mass (or a sum of dirac masses linked by a specific algebraic relation), the asymptotic behaviour was also defined thanks to the Mellin transform, but was more involved, with an infinite sum of contributions and a still slower polynomial rate of convergence.
\end{itemize}

Despite these  results, a question remains unclear: can we observe a kind of "oscillatory" behaviour, as in the case of a limit cycle~\cite{GreinerNagel,bernard:hal-01363549}?

In Proposition~\ref{prop:explicit}, we first provide an explicit solution of Equation~\eqref{eq:main} and discuss its interpretation, in particular in terms of possible periodicity. In Section~\ref{sec:asymp}, we investigate in more detail the asymptotic behaviour, based on the estimates obtained in~\cite{DE} and with the help of a rescaling inspired by~\cite{MR2017852}. 

\section{An explicit formulation}
Explicit formulations may be obtained, as said above, via the Mellin transform of the equation - the Mellin transform has also been used in other studies for the qualitative behaviour of solutions of equations of the same type, see for instance~\cite{BDE,vanbrunt_wake_2011,DET2017,Escobedo2016,Escobedo2017}.  Analytical solutions for specific cases of the eigenvalue problem have also been given in some studies, see e.g.~\cite{HallWake_1990,PR,DG} for some examples. 

Else, obtaining analytical solutions for the time-dependent equation is not frequent - let us mention~\cite{Stewart90,ZiffGrady} for the fragmentation equation, and~\cite{Zai} for the full analytical solution to the cell division equation with constant coefficients. Up to our knowledge, the following explicit solution for our case was not known.
\begin{prop}\label{prop:explicit}
Let ${\cal S}'(\R_+)$ the space of distribution functions (dual space of ${\cal S}(\R_+)$ the Schwartz space on $\R_+$) and  $u_0\in {\cal S}'(\R_+)$. The distribution defined in a weak sense by
\begin{equation}\label{def:u}
u(t,x)=e^{-(b+g)t} \sum\limits_{k=0}^\infty u_0(\alpha^k x e^{-gt}) \f{(b\alpha^2 t)^k}{k!},\qquad t>0,\quad x>0,
\end{equation}
 is solution to Equation~\eqref{eq:main}. Moreover, if $u_0\in L^p(x^qdx)$ then $u \in L^\infty(0,T;L^p(x^q dx))$ for any $T>0,$ $p\in [1,\infty]$ and $q\in \R.$ Similarly, if $u_0\in {\cal M}_+^b (x^qdx)$ the space of  nonnegative bounded measures absolutely continuous with respect to the measure $x^q dx,$ $u \in L^\infty(0,T;{\cal M}_+^b (x^qdx)).$ 
\end{prop}
\begin{proof}
The spaces to which $u(t,x)$ belongs to are immediate by using the definition~\eqref{def:u}, multiplying it by the convenient weight or test function, and make a term-by-term change of variables $y=\alpha^k x e^{-gt};$ it is linked to the fast convergence of the terms $\f{s^k}{k!}$ defining the series of the exponential.
The proof that $u(t,x)$ satisfies Equation~\eqref{eq:main} in a weak sense can be done similarly, by multiplying the equation applied to $u(t,x)$ by a test function, integrating by parts and making a change of variables.
\end{proof}
This formula makes directly appear several interesting features, linked to the very specific shape of the fragmentation kernel $k_0(z)=\alpha\delta_{z=\f{1}{\alpha}}.$ 
\begin{itemize}
\item At time $t=0^+,$ there is an immediate appearance of contributions to the distribution in $x$ of all the points $\alpha^k x$ with $k\geq 0.$ This can be interpreted in terms of division: without growth, cells of size $x$ contained in an interval $[x,x+dx]$ may come from $k$ times the division of cells of size $\alpha^k x,$ contained in the interval $[\alpha^kx, \alpha^k (x+dx)]$, this division producing $\alpha^k$ cells of size $x$ so that there is a factor $\alpha^{2k}$ in the contribution coming from the division of size $[\alpha^k x, \alpha^k x +\alpha^k dx]$ particles. Now, the probability of $k$ successive division of given cells of size $\alpha^k x$ in a time interval $\delta t$ is proportional to $\f{b^k (\delta t)^k}{k!},$ the product of $k$ times the probability for a cell to divide taken among an infinite possibility of divisions, and then renormalised by $e^{-b\delta t}$ to obtain a total probability equal to $1$ on all the possibilities to divide $0,$ $1,$ $\dots$ $k$ times. When time passes, this remains true, and the formula follows the characteristic lines $x e^{gt}: $ without the birth term on the right-hand side, the solution of the equation
$$\f{\p}{\p t} u +\f{\p}{\p x} (g x u(t,x))+bu(t,x)=0$$
would be $u(t,x)=e^{-(b+g)t}u_0(xe^{-gt}),$ the first term of the series; on the contrary, without the growth and the left-hand side division term, the equation
$$\f{\p}{\p t} u =b \alpha^2  u(t,\alpha x),\qquad u(0,x)=u^0(x)$$
would admit for solution $ \sum\limits_{k=0}^\infty u_0(\alpha^k x) \f{(b\alpha^2 t)^k}{k!}$ for $t>0$ and $x>0.$

\item Contrarily to other cases (with a smoother fragmentation kernel or a non-linear growth rate), we see that the equation has no smoothing effect. Taking for instance the case of a dirac initial data $u_0(x)=\delta_{x_0},$ we see that as expected intuitively the mass is permanently supported by a countable set of dirac masses, taking values along characteristic lines $x=\alpha^{-k}  x_0 e^{gt}$ representing, for $k=0,$ the ancestor characteristic curve, and for $k\geq 1,$ the characteristic line of the $k-th$ generation of offspring (individuals having divided $k$ times at time $t$).

\end{itemize}

Despite its simple formulation, the analytical formula~\eqref{def:u} does not lead directly to an easy  asymptotic behaviour. This was also the case with the formulation obtained in~\cite{DE} using  Mellin and inverse Mellin transform: a complete asymptotic analysis of the complex integral was necessary to obtain an asymptotic behaviour. 

However, an important  clue is given to the question of possible oscillations  when looking at the solution obtained for $u_0=\delta_{x_0}:$ the mass is permanently supported in the countable set of dirac at points $x=\alpha^{-k}e^{gt}x_0$ for $k\in \N$, so that at each period of time $t=nT$ such that $e^{gT}=\alpha,$ and only at them, the dirac masses come back to the points $x=\alpha^{-k+n}x_0$. This set tends to the countable set $x=\alpha^k x_0$ with $k\in \Z.$ But can we say more concerning the mass of each of these points? 

In a more general manner, if $supp (u_0)\subset [x_0,x_1]$, at any time one has $supp (u(t,\cdot))\subset \cup_{k\in\N} [2^{-k} x_0e^{gt},2^{-k} x_1 e^{gt}]$ for $k\in \N:$ if for instance $\f{1}{2}<x_0<x_1\leq 1,$ we have $supp\left(u(t,\cdot)\right) \cap \cup_{k\in\N} (2^{-k-1}e^{gt},2^{-k}x_0e^{gt})=\emptyset.$  It is thus clear that no \emph{pointwise} limit toward a steady behaviour is possible.

\section{Asymptotic behaviour}
\label{sec:asymp}

\subsection{Asymptotics via the Mellin transform~\cite{DE}}
Let us here assume $g=0$ and $b=1,$ denote $v(t,x)$ the corresponding solution of the pure fragmentation equation, we know that $u(t,x)=e^{-gt} v(bt,xe^{-gt})$, and Formula~\eqref{def:u} becomes
\begin{equation}\label{def:frag}
v(t,x)=e^{-t}\sum\limits_{k=0}^\infty u_0(\alpha^k x) \f{(\alpha^2 t)^k}{k!}.
\end{equation}
In~\cite{DE}, another explicit formula was obtained using Mellin and inverse Mellin transform.
For the sake of simplicity, we restrict ourselves to $u_0 \in {\cal C}^2_0 (\R_+),$ the space of two-times differentiable functions on $R_+$ decaying faster than any power law in $0$ and $+\infty$ (see Theorem~3.1. in~\cite{DE} for more general assumptions).

Define
\begin{equation}\label{def:Mellin}
K(s)=\int\limits_0^1x^{s-1} \alpha\delta_{x=\f{1}{\alpha}} dx=\alpha^{2-s},\qquad U_0(s)=\int\limits_0^\infty u_0(x)x^{s-1}dx,
\end{equation}
we have (Theorem~3.1. in~\cite{DE}), for 
$g=0$ and $b=1,$ and any $\nu\in \R:$
\begin{equation}\label{eq:Mellin}
v(t, x)=\frac {1} {2\pi i}\int \limits _{ \nu-i\infty }^{ \nu+i\infty }U_0(s)\,e^{(K(s)-1)t}x^{-s}ds.
\end{equation}
{ Following the notations of~\cite{DE}, the fragmentation kernel is in our case equal to $k(y,x)=\f{1}{y}k_0(\f{x}{y})$ with $k_0(z)=\alpha\delta_{z=\f{1}{\alpha}}.$ We see that $k_0$ satisfies the assumptions of Theorem 2.3.~(b) of~\cite{DE}, namely that it is a singular discrete measure whos support satisfies the Assumption H of~ \cite{DE} - since it is a unique point $\theta=\f{1}{\alpha}$.}
The following asymptotic formula was then obtained in Theorem 2.3. (b) of~\cite{DE} for $x<1$:
\begin{equation}\label{asymp1:v}
v(t, x)=x^{-s_+(t, x)}e^{(\alpha^{2-s_+(t, x)}-1)t}\frac{ \sum\limits_{ k\in \Z }U_0(s_k)
 e^{\frac {2i\pi k } {\log \alpha}\log x}}{{\sqrt{2\pi t} (\log\alpha) \alpha^{1-\f{s_+(t,x)}{2}} }}\left(1+o(t^{-\beta})\right),\end{equation}
for some $\beta>0$ and $s_+(t,x)$ defined by
$$s_+(t,x)=K'^{-1}\left(\f{\log x}{t}\right)=2-\f{\log\biggl(-\f{\log(x)}{t\log\alpha}\biggr)}{\log\alpha},\qquad 
x=e^{-t(\log\alpha)\alpha^{2-s_+}},\qquad
 s_k=s_+{-}\f{2ik\pi}{\log \alpha}.$$ 
Using the Poisson formula,  it was also noticed in Remark~3 of~\cite{DE} that it gives
\begin{equation}\label{asymp2:v}\begin{array}{ll}
v(t,x) &=  e^{(\alpha^{2-s_+(t, x)}-1)t} \frac{ \sum\limits_{ n\in \Z } u_0(\alpha^n x) \alpha^{s_+ n}}{{\sqrt{2\pi t}  \alpha^{1-\f{s_+(t,x)}{2}} }}\left(1+o(t^{-\beta})\right). 
\end{array}
\end{equation}
By a straightforward calculation, we can use this formula to obtain the asymptotic formulae for the general case $b,g>0,$ by the transformation $u(t,x)=e^{-gt}v(bt,xe^{-gt}).$ We take however here $b=1$ for the sake of simplicity, and still denote in short $s_+$ the function now defined in $s_+(t,xe^{-gt}).$ We have
\begin{equation}\label{mellin:asymp:growthfrag}\begin{array}{l}
u(t,x)\sim x^{-s_+(t, xe^{-gt})}e^{(\alpha^{2-s_+(t, xe^{-gt})}-1+g(s_+-1))t}\frac{ \sum\limits_{ k\in \Z }U_0(s_k)
 e^{\frac {2i\pi k } {\log \alpha}\log x}}{{\sqrt{2\pi t} (\log\alpha) \alpha^{1-\f{s_+(t,xe^{-gt})}{2}} }}(1+o(t^\beta)),
\\ \\
u(t,x) \sim  e^{(\alpha^{2-s_+(t, xe^{-gt})}-1-g)t} \frac{ \sum\limits_{ n\in \Z } u_0(\alpha^n xe^{-gt}) \alpha^{s_+ n}}{{\sqrt{2\pi t}  \alpha^{1-\f{s_+(t,xe^{-gt})}{2}} }}.
\end{array}
\end{equation}
Despite its resemblance with~\eqref{def:frag}, no immediate link appears. 

\subsection{Weak convergence result}
Can we use the formula~\eqref{asymp1:v} or~\eqref{asymp2:v} to make appear an oscillatory asymptotic behaviour? 
Following~\cite{MR2017852}, let us first focus on the following rescalings of $v:$
$$r(t,y):=te^{2ty}v(t,e^{ty})=t e^{2ty+gt} u(\f{t}{b},e^{(y+g)t}) ,\qquad y_0:=-\log\alpha,$$
$$\tilde r(t,z)=r(t,y_0+\f{\sigma z}{\sqrt{t}})\f{\sigma}{\sqrt{t}}, \qquad \qquad \sigma^2=K''(y_0)=(\log\alpha)^2.$$
Such a rescaling is motivated by the fact that the lines $x=e^{ty},$ with $y<0$ constant, correspond to the lines $s_+(t,x)$ constant in time, leading to a given asymptotic profile in~\eqref{asymp1:v} or~\eqref{asymp2:v}. Moreover it is such that the integral of $r$ and $\tilde r$ is preserved, \emph{i.e.} 
$$\int\limits_{-\infty}^\infty  r(t,y) dy = \int\limits_{-\infty}^\infty \tilde r(t,y) dy=\int\limits_0^\infty x v(t,x) dx=\int\limits_0^\infty x u_0(x) dx,\qquad \forall\; t\geq 0.$$
Theorem~1 in~\cite{MR2017852} apparently contradicts any oscillatory behaviour by stating the following weak convergence result, for which we sketch below an alternative proof using Formula~\eqref{asymp1:v}.
 \begin{prop}[Specific case of Theorem~1 in~\cite{MR2017852}]
Let  $u_0\in {\cal C}^2_0(\R_+)$.
$$r(t,\cdot) \rightharpoonup \delta_{-\log\alpha}U_0(2),\qquad
\tilde r(t,\cdot) \rightharpoonup U_0(2)G,$$ with $G(z)=\f{e^{-\f{\cdot^2}{2}}}{\sqrt{2\pi} }$ in a weak sense: for any bounded $C^1$ function  $\phi$ on $\R,$ we have
$$\int\limits_{-\infty}^{+\infty} \phi(y)r(t,y)dy \to U_0(2) \phi(-\log\alpha)\; \hbox{{\text{and}}}\; \int\limits_{-\infty}^{+\infty} \phi(z)\tilde r(t,z)dz \to U_0(2) \int\limits_{-\infty}^{+\infty} \phi(z) \f{e^{-\f{z^2}{2}}}{\sqrt{2\pi} }dz,$$
with $U_0(2)=\int\limits_0^\infty x u_0(x)dx$ the initial mass and $K'(2)=\log\alpha$. 
\end{prop}
\begin{proof}
First, for $y<0,$ let us denote $s_+(y):=s_+(t,e^{ty})=2-\f{\log (-\f{y}{\log\alpha})}{\log \alpha}$, which is independent of the time $t$. We also notice that $\alpha^{2-s_+(y)}=-\f{y}{\log \alpha}.$

In Corollary~1 of~\cite{DE}, this result has been obtained in the case where the kernel is not singular, so that instead of the infinite sum $\sum\limits_{k\in \Z} U_0 (s_k) e^{\f{2i\pi k}{\log\alpha}yt}$ there was only the term $U_0(s_+(y)).$ Hence to prove the result, it only remains to show that the terms with $U_0(s_k)$ vanish for $k\neq 0.$ Under our simpler assumption of $u_0\in {\cal C}^2_0(\R_+)$, we have for any continuous and bounded  test function $\phi (y)$ with $y\in \R$:
$$\int\limits_{-\infty}^{+\infty} \phi(y)r(t,y)dy=\int\limits_{-\infty}^{-A} + \int\limits_{-A}^{-\varepsilon} +\int\limits_{-\varepsilon}^{+\infty} 
 \phi(y)r(t,y)dy.$$
For $\varepsilon>0$ small enough and $A>0$ large enough fixed, the first and the third integrals are estimated as in the proof of Corollary~1 in~\cite{DE} (in the notations of~\cite{DE} we have $p_0=-\infty$ and $q_0=+\infty$ due to the fast decay of $u_0$ in $0$ and $\infty$):  $v(t,x)$ being exponentially decreasing in time on these interval, and using the regularity assumptions on the test function, these integrals go to zero. 
Let us call $I$ the second integral, where the mass concentrates, and use the asymptotic behaviour recalled above:
 \begin{equation*}
\begin{array}{ll}
I&:= \int\limits_{-A}^{-\varepsilon}  \phi(y)r(t,y)dy\\ \\&= \left(1+o(t^{-\beta})\right)\int\limits_{-A}^{-\varepsilon}\phi(y) t e^{2ty}
e^{-s_+(y)y t}e^{(\alpha^{2-s_+(y)}-1)t}\frac{ \sum\limits_{ k\in \Z }U_0\left(s_+(y) + \f{2ik\pi}{\log\alpha}\right)
 e^{\frac {2i\pi k } {\log \alpha}ty}}{{\sqrt{-2\pi t (\log\alpha)y}}} dy
\\ \\
&= \left(1+o(t^{-\beta})\right)\int\limits_{-A}^{-\varepsilon}\phi(y) t 
e^{\f{\log(-\f{y}{\log\alpha})}{\log \alpha} y t}e^{(-\f{y}{\log\alpha}-1)t}\frac{ \sum\limits_{ k\in \Z }U_0\left(s_+(y) + \f{2ik\pi}{\log\alpha}\right)
 e^{\frac {2i\pi k } {\log \alpha}ty}}{{\sqrt{-2\pi t (\log\alpha)y}}} dy.
\\ \\
&=\left(1+o(t^{-\beta})\right)\sum\limits_{k\in \Z} I_k.
\end{array}
\end{equation*}
The term for $k=0$ is the same as in Corollary~1 in~\cite{DE}:
$$\begin{array}{ll}I_0
&= \left(1+o(t^{-\beta})\right)\int\limits_{-A}^{-\varepsilon} \phi(y) \sqrt{t} e^{\Psi(y) t} \f{U_0\big(s_+(y)\big) }{\sqrt{-2\pi (\log\alpha)y}}dy \to_{t\to\infty} \phi(-\log\alpha)U_0(2),
\end{array}$$
using Laplace's method and the fact that
$$\Psi(y):=\f{\log(-\f{y}{\log\alpha})}{\log\alpha} y -\f{y}{\log\alpha} -1,\qquad \Psi'(y)=\f{\log(-\f{y}{\log\alpha})}{\log\alpha},\qquad \Psi''(y)=\f{1}{y\log\alpha}$$
 has a unique maximum at $y_0=K'(2)=-\log\alpha$, with $\Psi(-\log\alpha)=0$ and $\Psi''(-\log\alpha)=-\f{1}{(\log\alpha)^2}.$ For the terms $I_k$ with $k\neq 0$ we have
$$\begin{array}{ll}I_k&= \left(1+o(t^{-\beta})\right)
\int\limits_{-A}^{-\varepsilon}\phi(y) t e^{2ty}
e^{-s_+(t, e^{ty})y t}e^{(\alpha^{2-s_+(t, e^{ty})}-1)t}\frac{ U_0\left(s_+(t,e^{ty})+\f{2ik\pi}{\log\alpha}\right) e^{\frac {2i\pi k } {\log \alpha}ty}
 }{{\sqrt{-2\pi t (\log\alpha)y}}} dy\\ \\
&= \left(1+o(t^{-\beta})\right)\int\limits_{-A}^{-\varepsilon} \phi(y)\sqrt{t} e^{(\Psi(y) +\f{2ik\pi}{\log\alpha} y)t} \f{U_0\big(s_+(y)+\f{2ik\pi}{\log\alpha}\big) }{\sqrt{- 2\pi (\log\alpha) y}}dy \to_{t\to\infty} 0,
\end{array}$$
by  the stationary phase  approximation. The proof for convergence of $\tilde r$ is similar.
\end{proof}

\subsection{Pointwise oscillatory asymptotic behaviour}
Despite its seemingly contradiction, we see by the above proof that the steady convergence obtained does not necessarily contradict pointwise oscillations:  a weak convergence may  happen even for pointwise oscillatory solutions, oscillatory terms compensating each other when averaged by integration. 

The weak convergence results above have shed light on the line where the mass concentrates: the line $x=e^{ty_0}=\alpha^{-t}.$ Let us first keep the above seen change of variables $r(t,y)=te^{2ty}v(t,e^{ty})$, Formulae~\eqref{asymp1:v} and~\eqref{asymp2:v} become,  for $y=y_0-\log\alpha:$ 
\begin{equation*}
r(t,y_0)=t\alpha^{-2t} v(t, \alpha^{-t})= \sqrt{\f{t}{2\pi}}{ \sum\limits_{ k\in \Z }U_0(2+\f{2ik\pi}{\log\alpha})
 e^{-{2i\pi k }t}}\left(1+o(t^{-\beta})\right),\end{equation*}
or, using the Poisson formula
\begin{equation*}\begin{array}{ll}
t\alpha^{-2t} v(t,\alpha^{-t}) &= t \alpha^{-2t} \frac{ \sum\limits_{ n\in \Z } u_0(\alpha^{n-t}) \alpha^{2n}}{{\sqrt{2\pi t}  }}\left(1+o(t^{-\beta})\right) =  \sqrt{\f{t}{2\pi}}  { \sum\limits_{ n\in \Z } u_0(\alpha^{n-t}) \alpha^{2(n-t)}}\left(1+o(t^{-\beta})\right).
\end{array}
\end{equation*}
These two formulae make obvious the periodic behaviour, of period $T=1,$ of the quantity $\f{r(t,y)}{\sqrt{t}}$.
More generally, for a given $y<0$ fixed,  these two formulae also show that the function 
$$f_y(t)=\sqrt{t}e^{2ty-\Psi(y) t}v(t,e^{yt})$$ is periodic in time of period $T_y=-\f{\log\alpha}{y}.$ Since the exponential term $e^{\Psi(y)t}$ is maximal for $y=-\log\alpha,$ the line $(t,\alpha^{-t})$ dominates all the others - what explains the weak convergence result - but each of these lines follow a specific type of periodicity. The period is larger when $|y|$ is smaller, which corresponds to the lines $x=e^{yt}$ going more slowly to zero - other said, to the righ-hand side of the gaussian in $y$ (see also Figure~\ref{fig:2}). This is explained by the fact that the gaussian becoming wider and wider since its standard deviation is proportional to $\sqrt{t},$ the periodicity needs to be faster in the forefront (left-hand side of the gaussian in Figure~\ref{fig:2}) and slower after.
\subsection{Numerical illustration}
To simulate more easily the asymptotic behaviour, we define
$$n(t,y)=e^{2y}v(t,e^y),$$
which satisfies the following equation
\begin{equation}
\f{\p}{\p t} n(t,y) + n(t,y)=n(t,y+\log \alpha),\qquad n(0,y)=e^{2y}u_0(e^y).
\end{equation}
We choose as an initial condition for $n(0,y)$ a gaussian of mean zero and variance $\sigma^2.$ For $\sigma$ large, we do not observe any oscillations - exactly as in the previous numerical illustration of~\cite{DE} where we did not pay attention to the oscillatory phenomena. But for $\sigma$ small enough, clear oscillations appear and illustrate exactly the results.
In Figure~\ref{fig:1}, we take $\sigma=0.1$, $\alpha=2$ and draw the numerical value of $\sqrt{t} n(t,-t\log(2)),$  $\sqrt{t} n(t,-2t\log (2))e^{(2\log (2)-1)t}$ and $\sqrt{t} n(t,-\f{t}{2}\log (2))e^{\f{1}{2}(1-\log 2)t}$ which as expected exhibit oscillations of period $1$, $\f{1}{2}$ and $2$ respectively. In Figure~\ref{fig:2}, we show the time evolution of the rescaled profile $\sqrt{t}n(t,y):$ we clearly see the envelope shape of a gaussian appear and become wider and wider, whereas equally-wide peaks are inside the gaussian. { To illustrate the importance of the initial condition, we take in Figure~\ref{fig:3} and~\ref{fig:4} the same quantities with the same parameter values, except the standard deviation, there equal to 0.2. In Figure~\ref{fig:5} we took $\sigma=0.5:$ no oscillation is visible anymore. The shape of the initial condition has also an influence, as illustrated in Figures~\ref{fig:6} and~\ref{fig:7} where we took a Heaviside function in $[0.8,1]$: the shape is conserved when the profile oscillates. Due to the nonlinearity of the initial condition, contrarily to the gaussian initial data, the oscillations never totally disappear for a larger initial support, as shown in Figures~\ref{fig:8} and~\ref{fig:9} where the support is $[-1,0]$, then Figure~\ref{fig:11} where it is $[-5,0]$.}

\begin{figure}[ht]
\centering
\includegraphics[width=\textwidth]{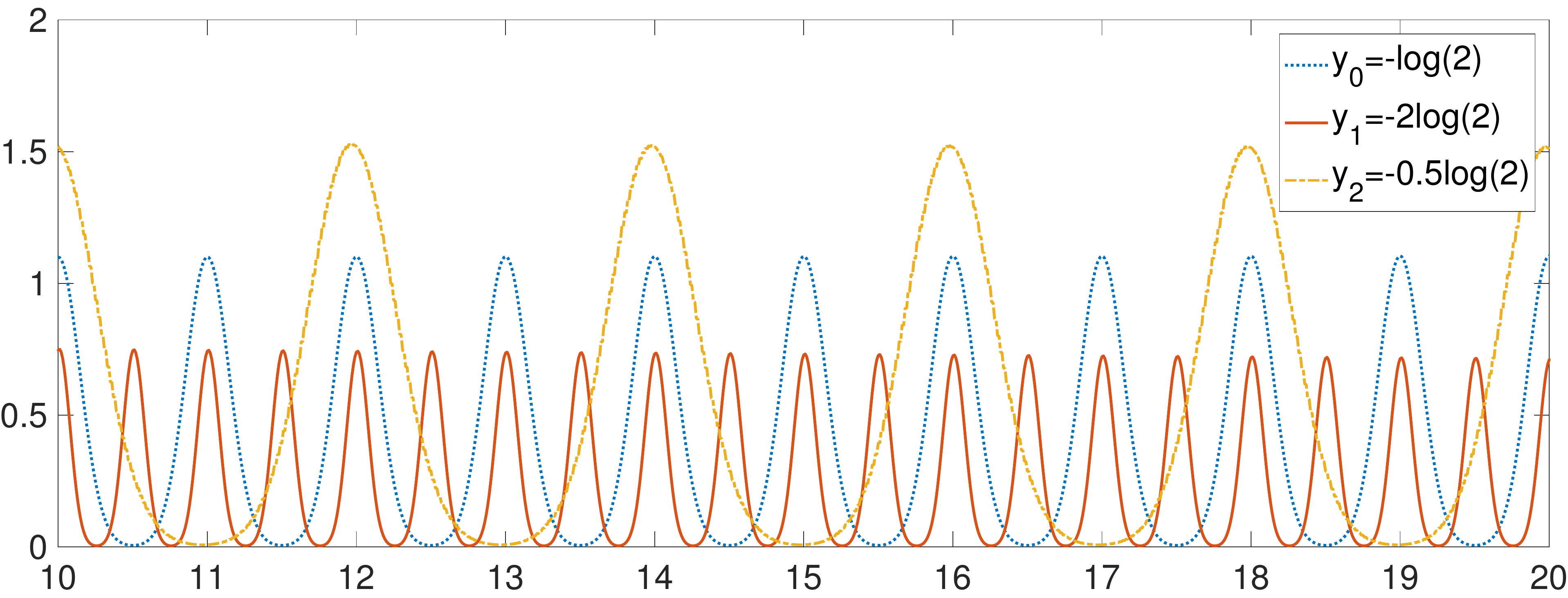}
\caption{\label{fig:1} Numerical simulation for $\alpha=2,$ $n(0,y)$ a gaussian of mean $0$ and standard deviation $0.1$.  Plot of the time variation of the quantity $\sqrt{t} e^{2ty-\Psi(y)t} v(t,e^{yt})=\sqrt{t} e^{-\Psi(y)t} n(t,yt)$ for $y=-\log (2),$ $y=-2\log (2)$ and $y=-0.5\log(2)$.}
\end{figure}

\begin{figure}[ht]
\centering
\includegraphics[width=\textwidth]{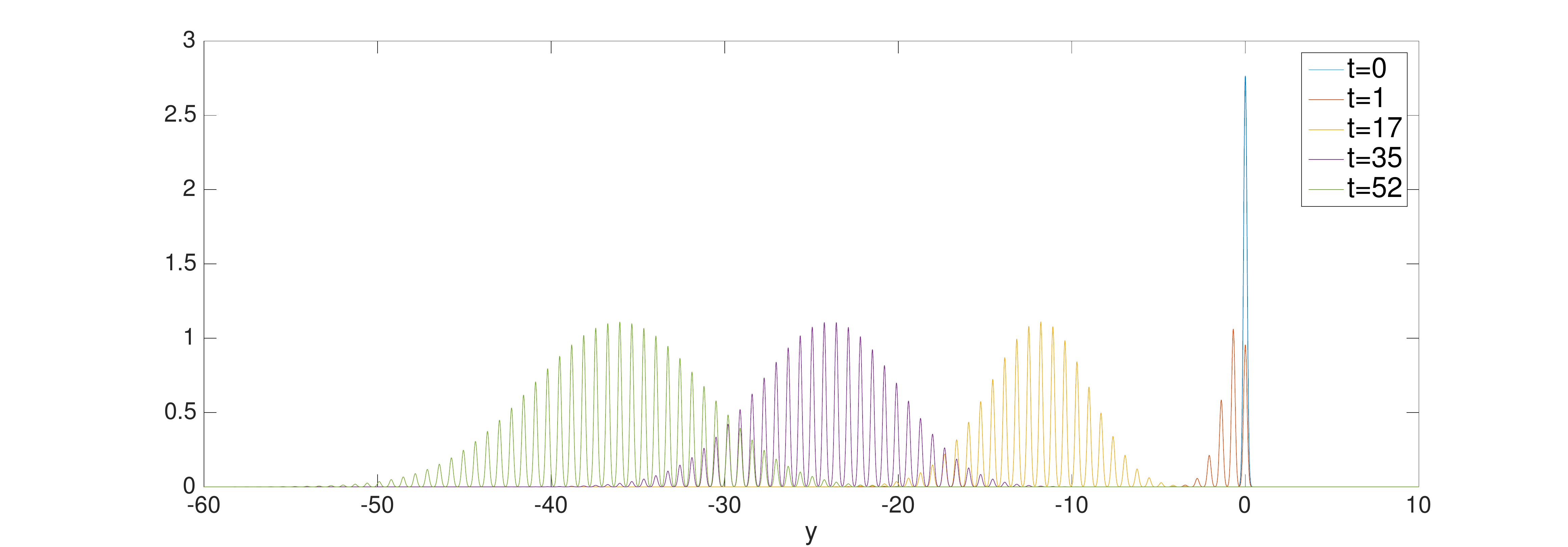}
\caption{\label{fig:2} Numerical simulation for $\alpha=2,$ $n(0,y)$ a gaussian of mean $0$ and standard deviation $0.1$.  Plot of the size-distribution of  $\sqrt{t} n(t,y)=\sqrt{t} e^{2y} v(t,e^{y})$. We see the shape of the gaussian becoming wider and wider, whereas oscillations are maintained.}
\end{figure}

\begin{figure}[ht]
\centering
\includegraphics[width=\textwidth]{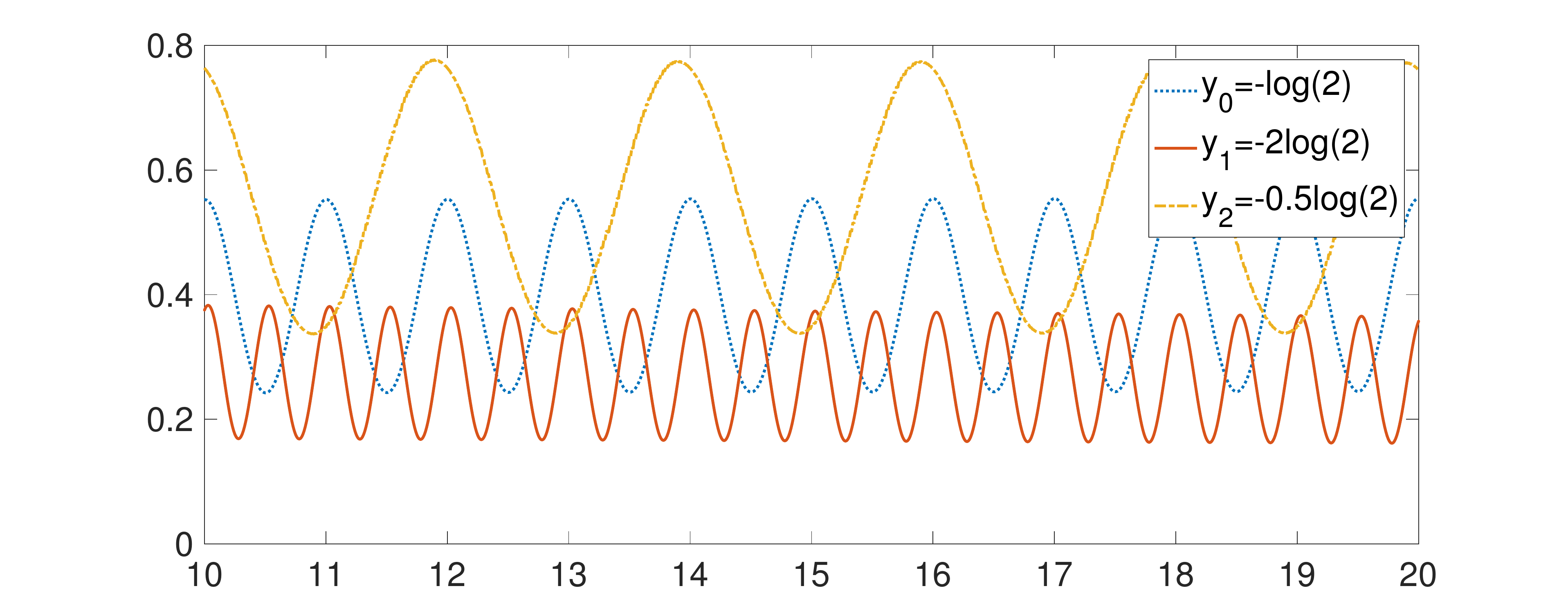}
\caption{\label{fig:3} {Numerical simulation for $\alpha=2,$ $n(0,y)$ a gaussian of mean $0$ and standard deviation $0.2$.  Plot of the time variation of the quantity $\sqrt{t} e^{2ty-\Psi(y)t} v(t,e^{yt})=\sqrt{t} e^{-\Psi(y)t} n(t,yt)$ for $y=-\log (2),$ $y=-2\log (2)$ and $y=-0.5\log(2)$.}}
\end{figure}

\begin{figure}[ht]
\centering
\includegraphics[width=\textwidth]{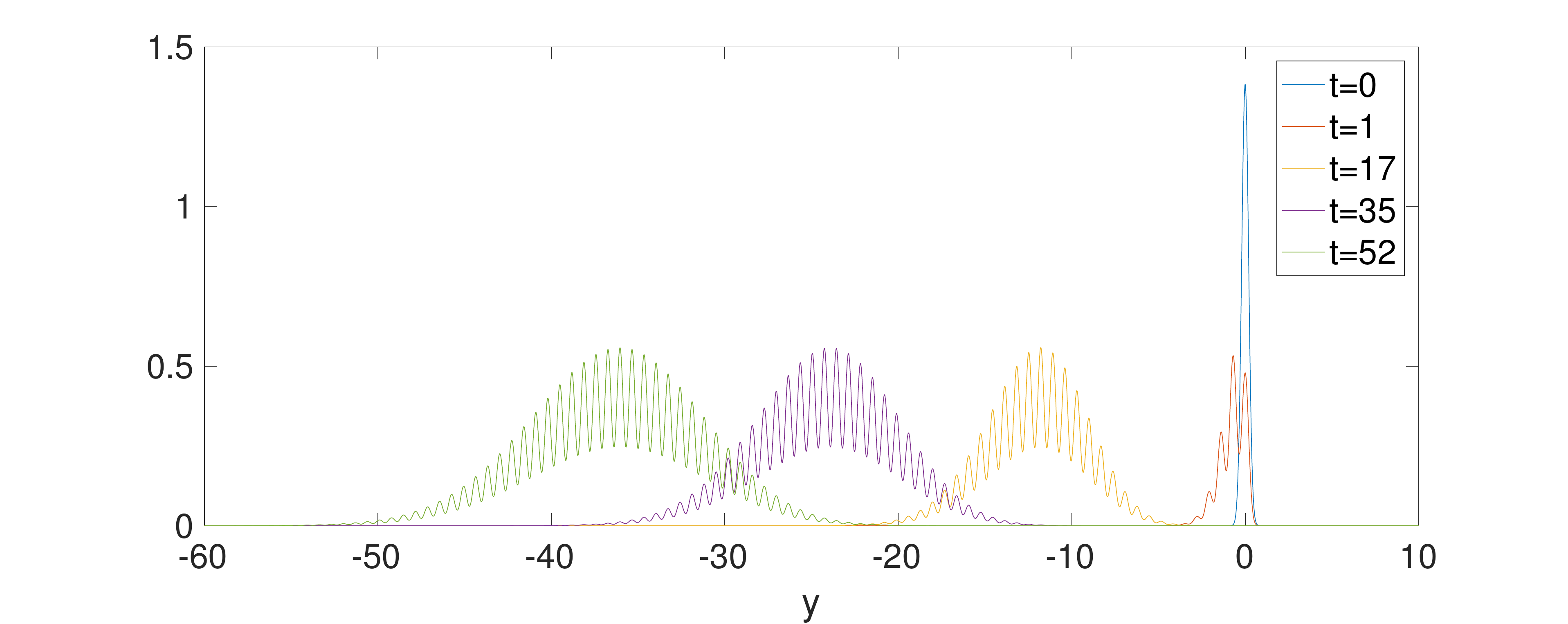}
\caption{\label{fig:4}{ Numerical simulation for $\alpha=2,$ $n(0,y)$ a gaussian of mean $0$ and standard deviation $0.2$.  Plot of the size-distribution of  $\sqrt{t} n(t,y)=\sqrt{t} e^{2y} v(t,e^{y})$. We see the shape of the gaussian becoming wider and wider, whereas oscillations are maintained but smaller than for $\sigma=0.1$.}}
\end{figure}

\begin{figure}[ht]
\centering
\includegraphics[width=\textwidth]{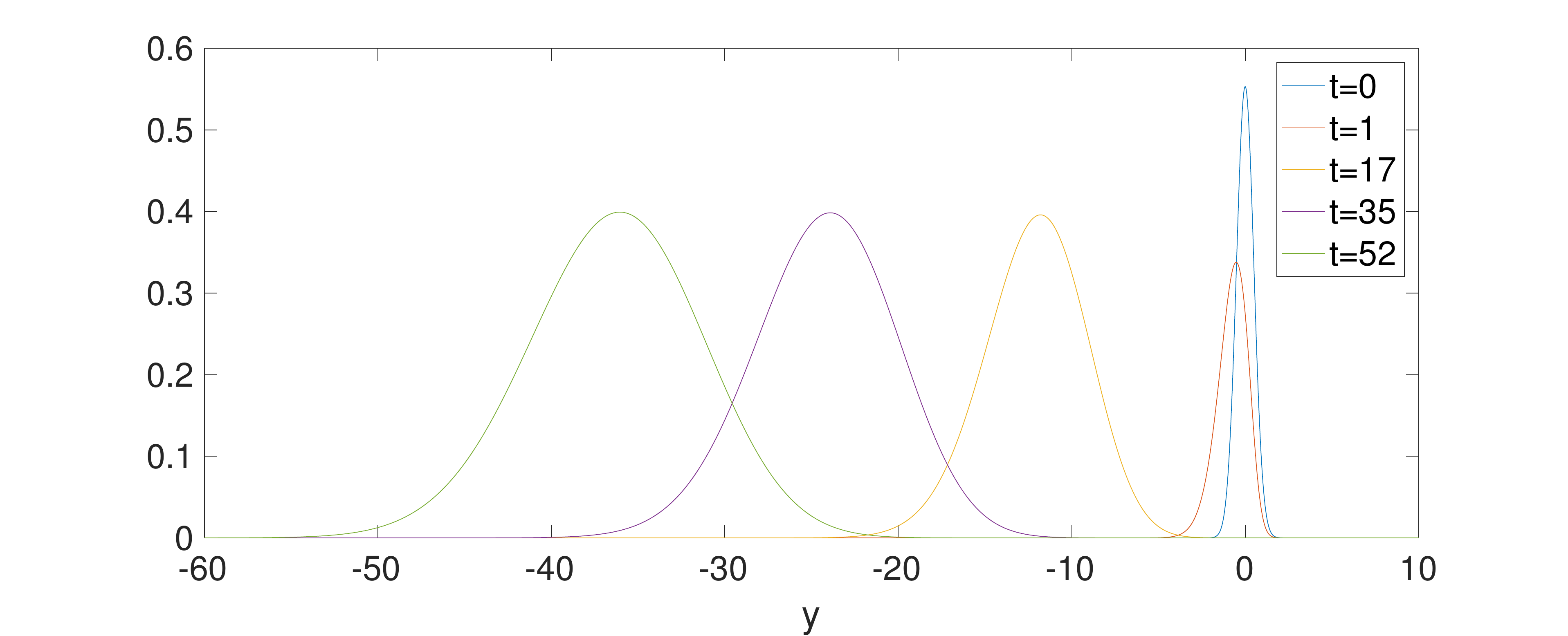}
\caption{\label{fig:5}{ Numerical simulation for $\alpha=2,$ $n(0,y)$ a gaussian of mean $0$ and standard deviation $0.5$.  Plot of the size-distribution of  $\sqrt{t} n(t,y)=\sqrt{t} e^{2y} v(t,e^{y})$. We see the shape of the gaussian becoming wider and wider, and no oscillation is anymore visible.}}
\end{figure}

\begin{figure}[ht]
\centering
\includegraphics[width=\textwidth]{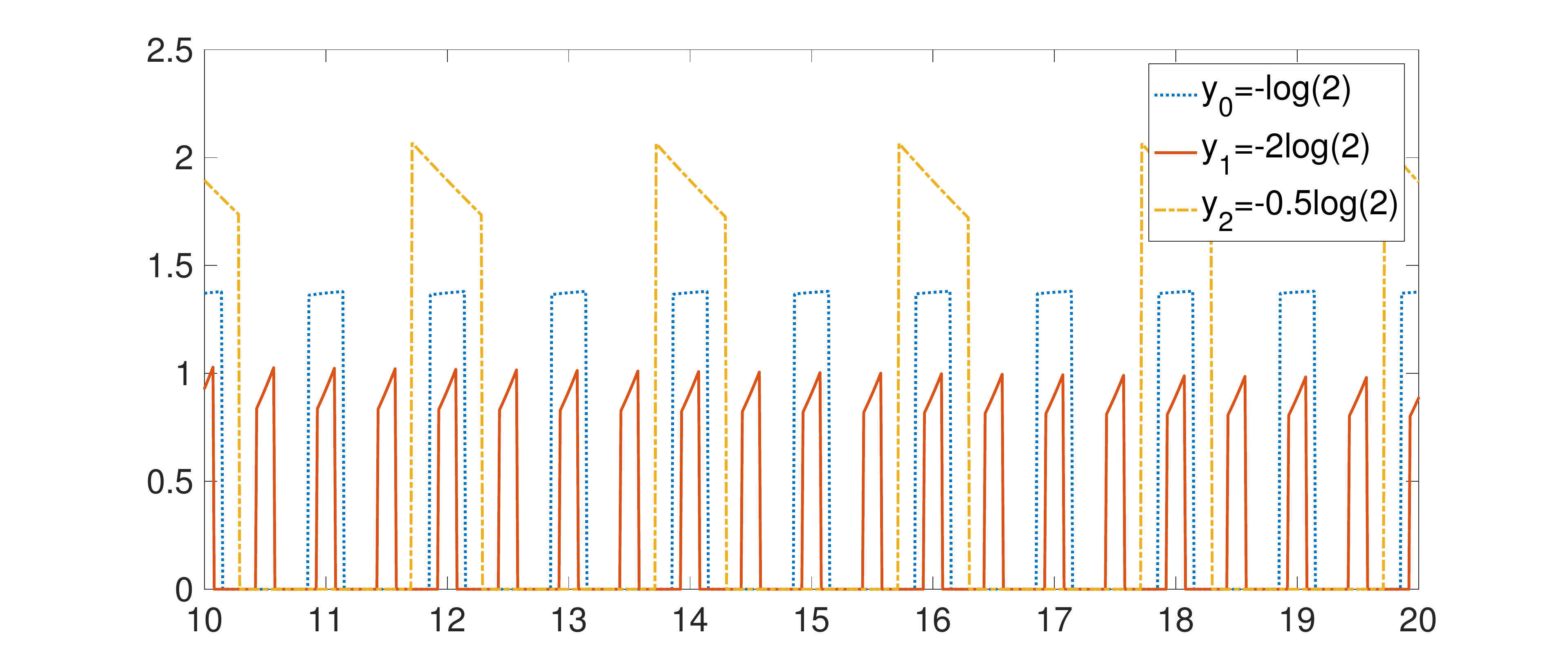}
\caption{\label{fig:6} {Numerical simulation for $\alpha=2,$ $n(0,y)$ a Heaviside on $[-0.2,0]$.  Plot of the time variation of the quantity $\sqrt{t} e^{2ty-\Psi(y)t} v(t,e^{yt})=\sqrt{t} e^{-\Psi(y)t} n(t,yt)$ for $y=-\log (2),$ $y=-2\log (2)$ and $y=-0.5\log(2)$.}}
\end{figure}

\begin{figure}[ht]
\centering
\includegraphics[width=\textwidth]{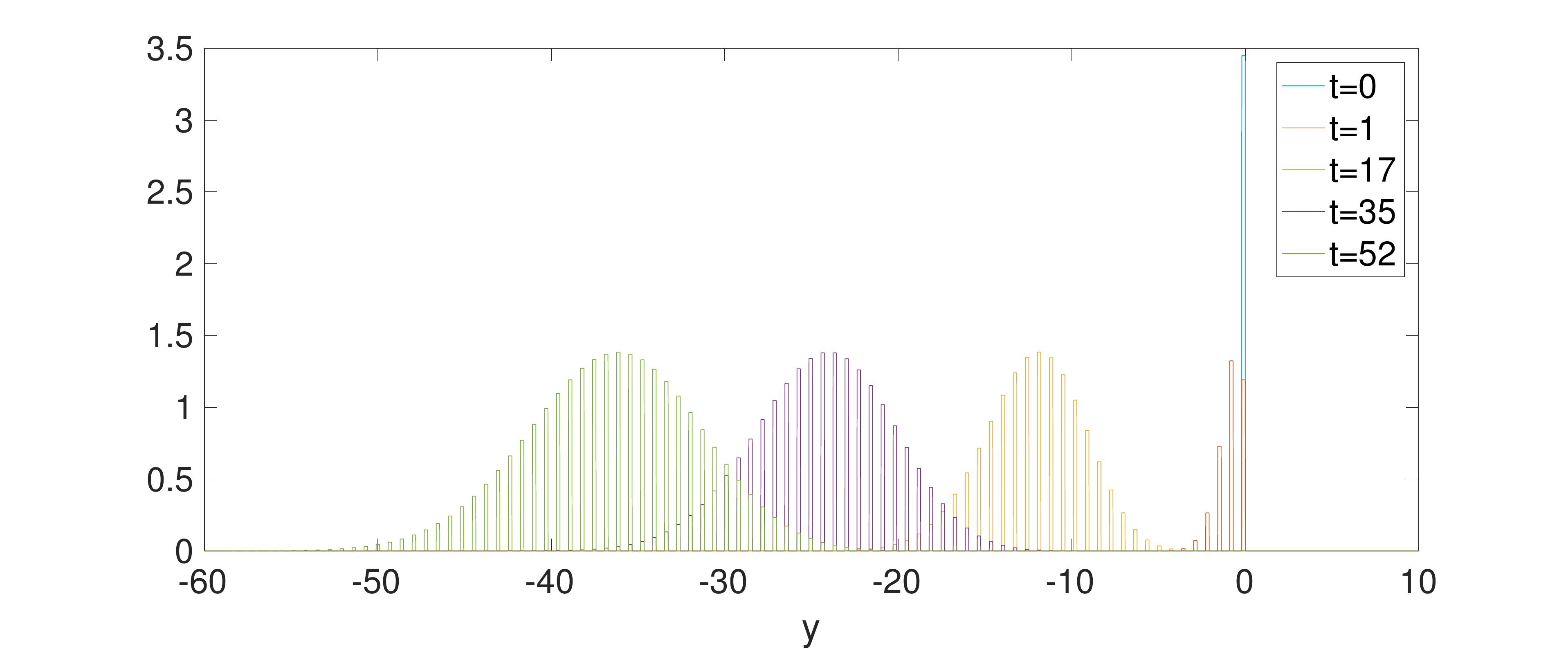}
\caption{\label{fig:7} { Numerical simulation for $\alpha=2,$ $n(0,y)$ a Heaviside on $[-0.2,0].$  Plot of the size-distribution of  $\sqrt{t} n(t,y)=\sqrt{t} e^{2y} v(t,e^{y})$. We see the shape of the gaussian becoming wider and wider, whereas oscillations are maintained and keep the shape of the Heaviside.}}
\end{figure}

\begin{figure}[ht]
\centering
\includegraphics[width=\textwidth]{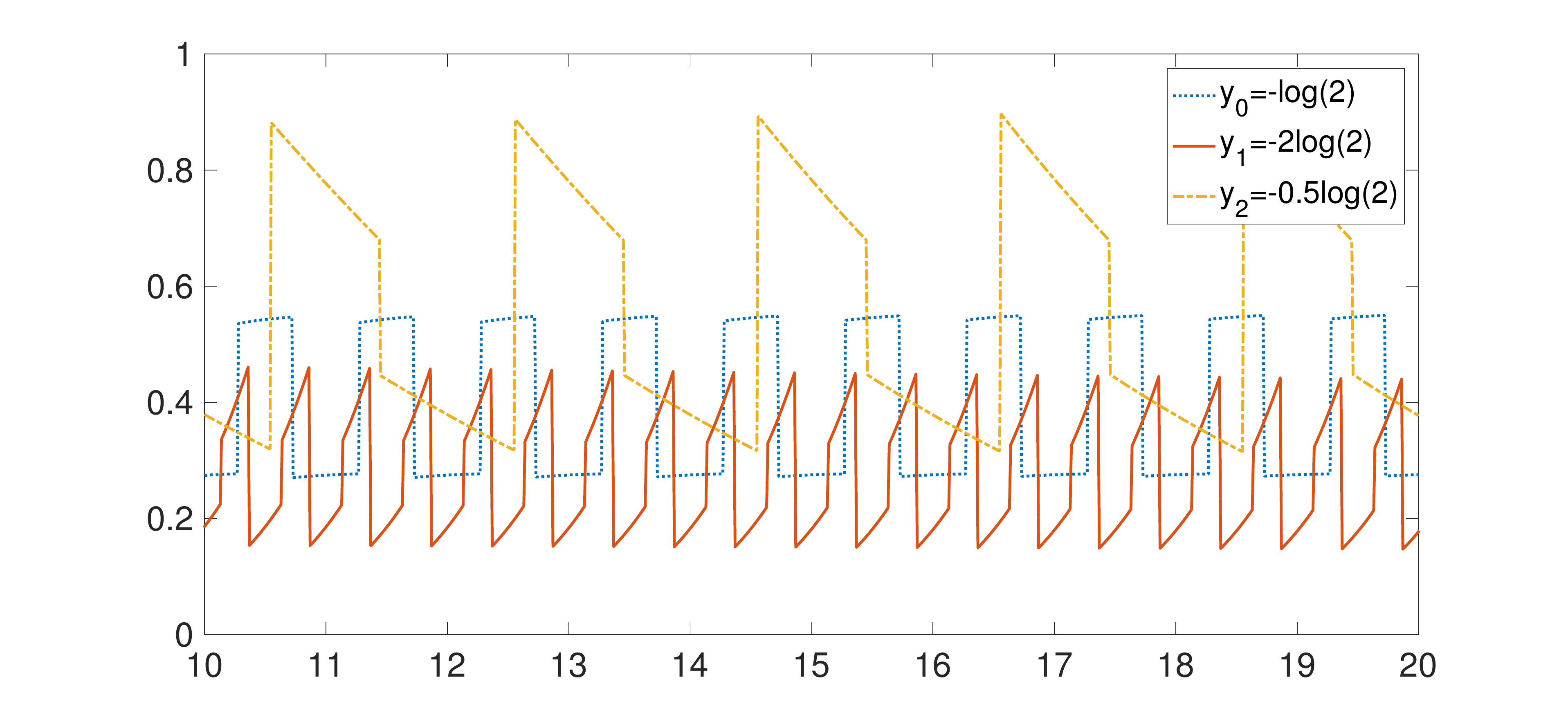}
\caption{\label{fig:8} {Numerical simulation for $\alpha=2,$ $n(0,y)$ a Heaviside on $[-1,0]$.  Plot of the time variation of the quantity $\sqrt{t} e^{2ty-\Psi(y)t} v(t,e^{yt})=\sqrt{t} e^{-\Psi(y)t} n(t,yt)$ for $y=-\log (2),$ $y=-2\log (2)$ and $y=-0.5\log(2)$.}}
\end{figure}

\begin{figure}[ht]
\centering
\includegraphics[width=\textwidth]{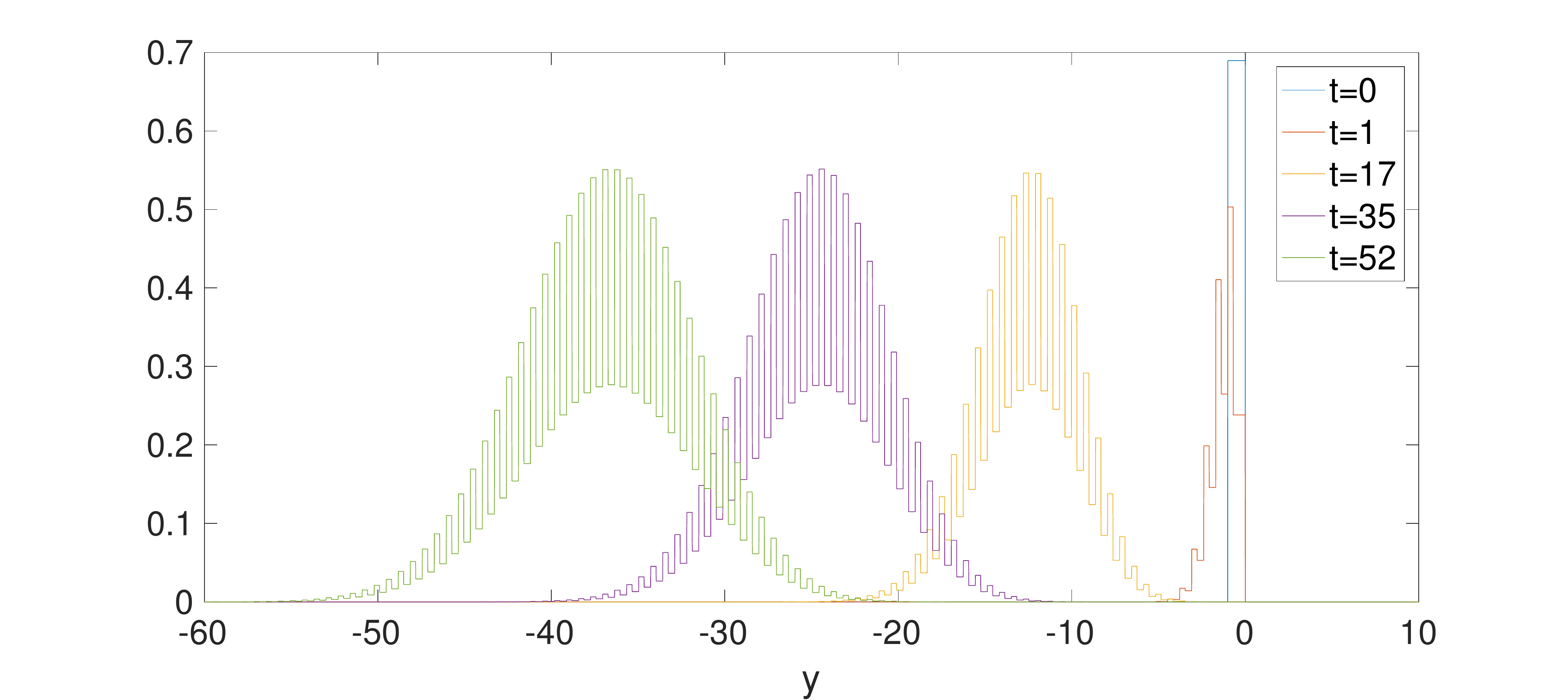}
\caption{\label{fig:9} { Numerical simulation for $\alpha=2,$ $n(0,y)$ a Heaviside on $[-1,0].$  Plot of the size-distribution of  $\sqrt{t} n(t,y)=\sqrt{t} e^{2y} v(t,e^{y})$. We see the shape of the gaussian becoming wider and wider, whereas oscillations are maintained (though smaller than for more peaked initial data) and keep the shape of the Heaviside.}}
\end{figure}

\begin{figure}[ht]
\centering
\includegraphics[width=\textwidth]{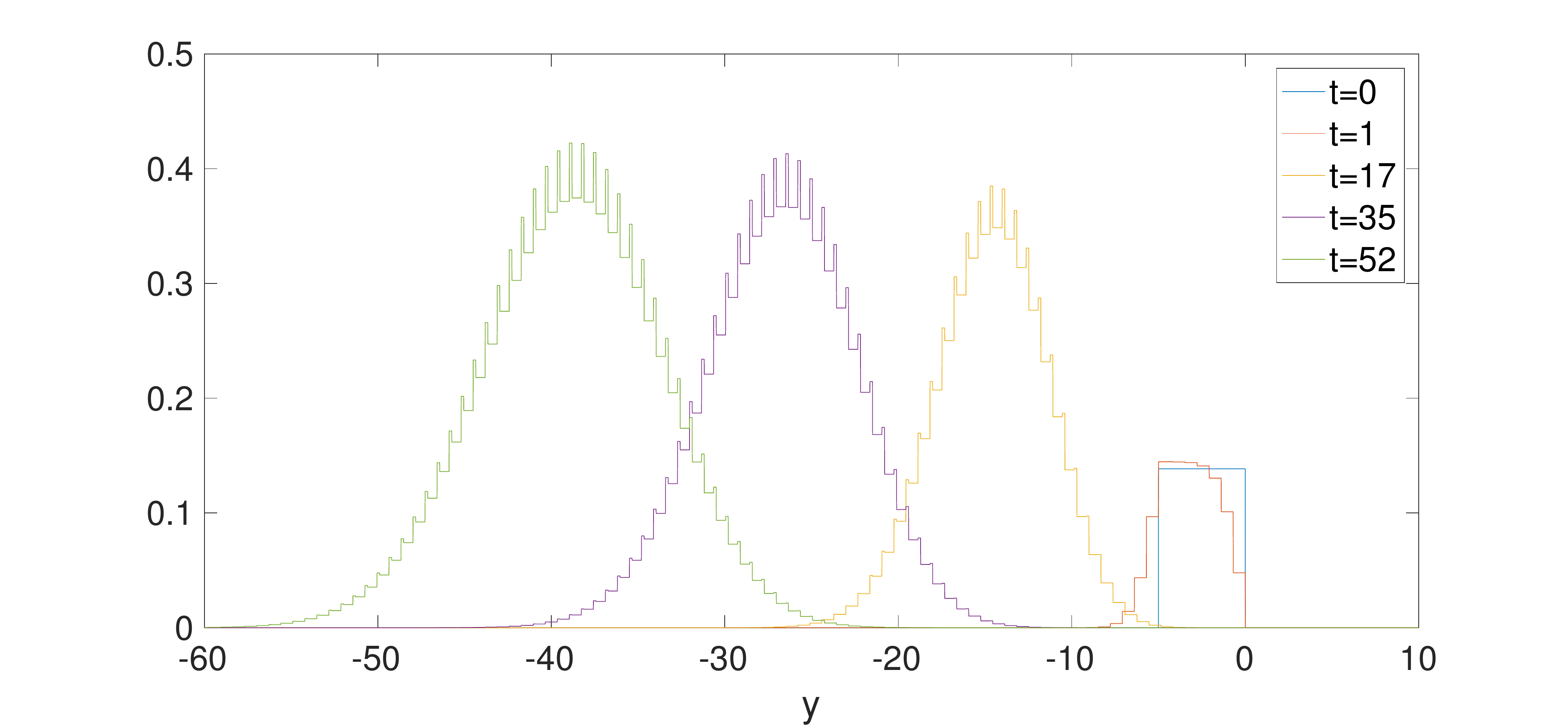}
\caption{\label{fig:11} { Numerical simulation for $\alpha=2,$ $n(0,y)$ a Heaviside on $[-5,0].$  Plot of the size-distribution of  $\sqrt{t} n(t,y)=\sqrt{t} e^{2y} v(t,e^{y})$. We see that the oscillations are maintained (though smaller than for more peaked initial data) and keep the shape of the Heaviside.}}
\end{figure}

\

{\bf Acknowledgments.} M.D. has been supported by the ERC Starting Grant SKIPPER$^{AD}$ (number 306321). The authors thank J. Bertoin, M. Escobedo and P. Gabriel for illuminating discussions, and M. Dauhoo, L. Dumas and P. Gabriel for the opportunity to work together at the CIMPA school in Mauritius.

\end{document}